\documentclass[12pt,a4paper]{amsart}
\usepackage{amsmath,amsfonts,amssymb,amsthm,amscd}
\usepackage{dsfont,mathrsfs,enumerate}
\usepackage{curves,epic,eepic}

\usepackage[colorlinks=true,
linkcolor=black, citecolor=blue, pagebackref]{hyperref}

\usepackage{url}
\usepackage{cancel}
\usepackage{pdflscape}

\usepackage[pdftex]{graphicx}
\usepackage{psfrag}
\usepackage{epic}
\usepackage{eepic}
\usepackage[matrix,arrow,curve]{xy}
\input{epsf}
\usepackage{tikz}
\usetikzlibrary{calc}
\usetikzlibrary{intersections}
\usetikzlibrary{through}
\usetikzlibrary{backgrounds}

\newtheorem{thm}{Theorem}[section]
\newtheorem{lemma}[thm]{Lemma}
\newtheorem{prop}[thm]{Proposition}
\newtheorem{coro}[thm]{Corollary}

\theoremstyle{definition}
\newtheorem{defi}[thm]{Definition}
\newtheorem{rem}[thm]{Remark}

\def\Z{\mathds Z}
\def\Q{\mathds Q}
\def\R{\mathds R}

\def\phi{\varphi}
\def\<{{\langle}}
\def\>{{\rangle}}

\newcommand{\lw}[1]{\mathrm{lw}(#1)}

\newcommand{\interior}[1]{\mathrm{int}(#1)}

\newcommand{\Min}[2]{\mathrm{Min}_{#1}(#2)}
\newcommand{\conv}[1]{\mathrm{conv}\left(#1\right)}
\newcommand{\normalfan}[1]{\Sigma_{#1}}

\definecolor{zzttqq}{rgb}{0.6,0.2,0}
\definecolor{ududff}{rgb}{0.3,0.3,1}
\definecolor{aabbcc}{rgb}{1,0,0}

\begin{document}
	
	\title[Lattice width $2$ and corresponding toric hypersurfaces]{Lattice $3$-polytopes of lattice width $2$ and corresponding toric hypersurfaces}
	
	\author[Martin Bohnert]{Martin Bohnert}
	\address{Mathematisches Institut, Universit\"at T\"ubingen,
		Auf der Morgenstelle 10, 72076 T\"ubingen, Germany}
	\email{martin.bohnert@uni-tuebingen.de}
	
	\begin{abstract}
	The Kodaira dimension of a nondegenerate toric hypersurface can be computed from the dimension of the Fine interior of its Newton polytope according to recent work of Victor Batyrev, where the Fine interior of the Newton polytope is the subpolytope consisting of all points which have an integral distance of at least $1$ to all integral supporting hyperplanes. In particular, if we have a Fine interior of codimension $1$, then the hypersurface is of general type and the Newton polytope has lattice width $2$. In this article we study this situation for lattice $3$-polytopes and the corresponding surfaces of general type. In particular, we classify all $2$-dimensional Fine interiors of those lattice $3$-polytopes which have at most $40$ interior lattice points, thus obtaining many examples of surfaces of general type and genus at most 40.
	\end{abstract}
	
	\maketitle
	
	
	\thispagestyle{empty}
	
\section{Introduction}

Let $M\cong \Z^d$ be a lattice of rank $d\in \Z_{\geq 1}$ and $N=\mathrm{Hom}(M\Z)\cong \Z^d$ the dual lattice. We have the natural pairing $\left<\cdot,\cdot\right>\colon M \times N \to \Z$, which extends to the real vector spaces $M_\R= M\otimes \R$ and $N_\R= N\otimes \R$. For a rational $d$-polytope $P$ in  $M_\R$, i.e. $P\subseteq M_\R$ is the convex hull of a finite number of points of $M_\Q=M \otimes \Q$ and $P$ is of full dimension, we have the \textit{Fine interior} of $P$ defined as
\begin{align*}
	F(P) := \{x \in M_\R \mid \left<x,n\right> \geq \Min{P}{n} + 1 \ \forall n \in N\setminus \{0\} \} \subseteq P
\end{align*}
with
\begin{align*}
	\Min{P}{n}:=\min \{\left<x,n\right> \mid x \in P\}.
\end{align*}
The Fine interior is itself a rational polytope, and if $P$ is the Newton polytope of a nondegenerate toric hypersurface $Z$ and has a non-empty Fine interior, then recent work by Victor Batyrev \cite{Bat23} shows that the Fine interior and additional combinatorial data can be used to construct a minimal model $\hat{Z}$ of the hypersurface. Moreover, we can compute some important invariants of the minimal model directly from the Fine interior. The Kodaira dimension is given by \cite[9.2]{Bat23} as
\begin{align*}
	\kappa(\hat{Z})=\min \{\dim F(P), d-1 \}
\end{align*}
and, for example, for the top intersection number in the case $\dim F(P)=d-1$ we have the formula (\cite[5.1]{Gie22}, \cite[9.4]{Bat23})
\begin{align*}
	(K_{\hat{Z}})^{d-1}= 2 \mathrm{Vol}_{d-1}(F(P)),
\end{align*}
where $\mathrm{Vol}_{d-1}$ means that we are looking at the normalized volume in the affine subspace generated by $F(P)$, where normalized means that we have volume $1$ for a simplex whose vertices form an affine lattice basis of the sublattice in the subspace.

We focus in this article on the case $d=3$ with $\dim F(P)=2$, so our nondegenerate toric hypersurface becomes a surface of general type. This implies that our Newton polytope $P\subseteq M_\R$ is a lattice $3$-polytope with lattice width $2$, where we should remember that the \textit{lattice width} $\lw{P}$ is given by
\begin{align*}
	\lw{P} := \min\{ -\Min{P}{-n} - \Min{P}{n} \mid n \in N \setminus \{0  \}.
\end{align*}
and we call $n_{lw}\in N$ a \textit{lattice width direction} if $\lw{P}=-\Min{P}{-n_{lw}} - \Min{P}{n_{lw}}$. 

Why is the combinatorics in this situation easy enough to handle? As a first step, we will see in Theorem \ref{Fine_int_by_middle_polytope} in section $2$, that the Fine interior of a lattice $d$-polytope with lattice width $2$ and lattice width direction $n_{lw}$ is completely determined by its half-integral \textit{middle polytope}, i.e. the rational polytope of dimension $d-1$, which we get by the intersection of $P$ with the hyperplane 
\begin{align*}
	\{x \in M_\R \mid \left<x,n_{lw}\right> = \Min{P}{n_{lw}}+1\}.
\end{align*}
In particular, we can work for lattice $3$-polytopes with a \textit{middlepolygon} and only have to do combinatorics from there on only in dimension $2$. As a second step, we will see in section $3$ that in this case not only the middle polygon but also the Fine interior is a half-integral polygon (Theorem \ref{Fine_int_half_integral}). Moreover, we will see there that the Fine interiors can always be described by the union of a lattice polygon and some very special triangles that we can have on edges of the lattice polygons having two lattice points (see Theorem \ref{hats}). It turns out that this description of the Fine interior is restrictive enough that we end up with an efficient classification algorithm for these Fine interiors in section $4$. We will use this algorithm there to obtain our main classification result in \ref{classification}: There are up to affine unimodular equivalence exactly $24 324 158$ different $2$-dimensional Fine interiors those lattice $3$-polytopes which have at most $40$ interior lattice points. In the last section we use this classification to classify as a corollary the Chern numbers for nondegenerated toric hypersurfaces of genus at most $40$ which have as Newton polytope a lattice $3$-polytope with $2$-dimensional Fine interior. Having the Chern numbers of the surfaces, we can arrange them in the geography of the surfaces of general type.
	
\section{The Fine interior of lattice polytopes of lattice width 2}

In this section we will see that the Fine interior of a lattice $d$-polytope of lattice width $2$ for arbitrary dimension $d$ is determined by its half-integral middle polytope, which allows us to do our Fine interior computations in codimension $1$, or in other words, without loss of generality, we can restrict ourselves for Fine interior computations to the case of pyramids of heigth $2$ which have the same middle polytope. 

We begin with the following lemma, which shows why the case of lattice width $2$ is important for Fine interior computations. 

\begin{lemma}
Let $P\subseteq M_\R$ be a rational $d$-polytope. If $\lw{P}<2$, then $F(P)=\emptyset$. If $\lw{P}=2$, then $\dim F(P)<\dim P$, and if $\dim F(P)=\dim(P)-1$, then $\lw{P}=2$.
\end{lemma}
\begin{proof}
	If $n_{lw}\in N$ is a lattice width direction, then we have
	\begin{align*}
		F(P) \subseteq \{x\in M_\R \mid \left<x,n_{lw}\right> \geq \Min{P}{n_{lw}}+1, \left<x,-n_{lw}\right> \geq \Min{P}{-n_{lw}}+1 \}
	\end{align*}
	and the set on the right side is empty for $\lw{P}<2$ and at of dimension $d-1$ for $\lw{P}=2$. If we have $\dim F(P)=\dim(P)-1$, then we must have a normal vector $n\in N$ for $F(P)$ such that $-n$ is also a normal vector and thus $-\Min{P}{-n}-\Min{P}{n}=2$. This implies $\lw{P}\leq 2$ and we have $\lw{P}\geq 2$ since this is always the case for non-empty Fine interior by the first part of the lemma.
\end{proof}

\begin{rem}
There are many examples of lattice polytopes $P\subseteq M_\R$ with lattice width greater than $2$ and $\dim F(P) < \dim P-1$. Since $\dim F((d+1)\Delta_d)=0$ for the standard simplex $\Delta_d$, i.e. the convex hull of an affine lattice basis, and $\lw{(d+1)\Delta_d}=d+1$, we have for $d, k, w\in \Z, d\geq 2, 0\leq k<d-1, 3\leq w\leq d-k+1$ the lattice $d$-polytope $(d-k+1)\Delta_{d-k} \times [0,w]^{k}$ with lattice width $w>2$ and $\dim(F(P))=k$.
\end{rem}

\begin{rem}
We can see that $\dim(F(P))=d-1$ implies $\lw{P}\leq 2$ also from the general theory about lattice projections along the Fine interior. By \cite[9.1]{Bat23} we have a lattice projection along the Fine interior on a lattice polytope with a Fine interior of dimension $0$. So from a Fine interior of dimension $d-1$ we get a projection onto a lattice polytope of dimension $1$ with a Fine interior of dimension $0$. But up to lattice translation there is only $[-1,1]$ with this property, and a lattice projection onto $[-1,1]$ is obviously equivalent to $\lw{P}\leq 2$.
\end{rem}

We now rearrange our coordinates for a good working setup with the middle polytope and split the dual lattice $N$ with respect to the middle polytope. 

\begin{defi}
Let $P\subseteq \R\times M_\R$ be a lattice polytope with $\lw{P}=2$ and $P \subseteq [-1,1] \times M_\R$. Then we have the half-integral middle polytope $P_0\subseteq M_\R$ of $P$ defined by 
\begin{align*}
\{0\} \times P_0 = P \cap (\{0\} \times M_\R).
\end{align*}
Using $P_0$, we define a partition $N\setminus \{0\}=N_1(P_0)\cup N_2(P_0)$ of the non-zero dual lattice vectors by 
\begin{align*}
N_1(P_0) := \{n \in N\setminus \{0\} \mid \Min{P_0}{n} \in \Z\}, N_2(P_0) := \{n \in N\setminus \{0\} \mid \Min{P_0}{n} \notin \Z\}.
\end{align*}
\end{defi}

Now we are ready to show how the middle polytope determines the Fine interior of a lattice polytope of lattice width $2$.

\begin{thm}\label{Fine_int_by_middle_polytope}
Let $P\subseteq \R \times M_\R$ be a lattice polytope of lattice width $2$, $P \subseteq [-1,1] \times M_\R $ and $P_0\subseteq M_\R$ the middle polytope of $P$. Then the Fine interior $F(P)$ of $P$ is
\begin{align*}
F(P) = F_1(P)\cap F_2(P)
\end{align*}
with
\begin{align*}
F_1(P):=&\{0\} \times \{x \in M_\R \mid \left<x,n\right> \geq \Min{P_0}{n}+1 \ \forall n \in N_1(P_0)\} \\
F_2(P):=&\{0\} \times \{x \in M_\R \mid \left<x,2n\right> \geq \Min{P_0}{2n}+1 \ \forall n \in N_2(P_0)\} .
\end{align*}
In particular, $F(P)$ is completely determined by $P_0$.
\end{thm}
\begin{proof}
	Since $P\subseteq [-1,1] \times M_\R$ and $\lw{P}=2$, we have $F(P)=F(P)\cap (\{0\} \times M_\R)$ and 
	\begin{align*}
	F(P)=\{0\} \times \{x \in M_\R \mid \left<(0,x),(n_0,n)\right> \geq \Min{P}{(n_0,n)}+1 \ \forall n_0 \in \Z, n \in N\setminus \{0\}\}.
	\end{align*}
	So we have $F(P)\subseteq \{0\} \times P_0$, and since $P$ is a lattice polytope and $P_0$ is half-integral, it is sufficient enough to look at those pairs $(n_0,n)$ with
	\begin{align*}
		\Min{P}{(n_0,n)}+\frac{1}{2}\geq \Min{\{0\} \times P_0}{(n_0,n)}.
	\end{align*}
	With
	\begin{align*}
	\tilde{N_1} :=& \{(n_0,n) \in \Z \times (N\setminus \{0\}) \mid \Min{P}{(n_0,n)}= \Min{\{0\} \times P_0}{(n_0,n)}\}, \\\tilde{N_2} :=& \{(n_0,n) \in \Z \times (N\setminus \{0\}) \mid \Min{P}{(n_0,n)}+\frac{1}{2}= \Min{\{0\} \times P_0}{(n_0,n)}\}
	\end{align*}
	we get
	\begin{align*}
	F(P)=& (\{0\} \times \{x \in M_\R \mid \left<(0,x),(n_0,n)\right> \geq \Min{\{0\} \times P_0}{(n_0,n)}+1 \ \forall (n_0,n) \in \tilde{N_1}\}) \ \cap \\
	& (\{0\} \times \{x \in M_\R \mid \left<(0,x),(n_0,n)\right> \geq \Min{\{0\} \times P_0}{(n_0,n)}+\frac{1}{2} \ \forall (n,\nu) \in \tilde{N_2}\}\\
	=& (\{0\} \times \{x \in M_\R \mid \left<x,n\right> \geq \Min{P_0}{n}+1 \ \forall (n_0,n) \in \tilde{N_1}\} ) \ \cap \\
	& (\{0\} \times \{x \in M_\R \mid \left<x,2n\right> \geq \Min{P_0}{2n}+1 \ \forall (n_0,n) \in \tilde{N_2}\} )\\
	\end{align*}
	We have for $(n_0,n)\in \tilde{N_1}$ that $\Min{P_0}{n}=\Min{P}{(n_0,n)}\in \Z$ and so $n\in N_1(P_0)$. With the same argument we get $n\in N_2(P_0)$ for all $(n_0,n)\in \tilde{N_2}$. It remains to show that for all $n\in N_1(P_0)$ we have a $n_0\in \Z$ with $(n_0,n)\in \tilde{N_1}$ and analogously for $N_2(P_0)$ and $\tilde{N_1}$.
	
	Let be $n\in N_1(P_0)$, i. e. $\Min{P_0}{n}\in \Z$. Then there is a vertex $e$ of $P_0$ with $\Min{P_0}{n}=\left<e,n\right>$. We choose such a vertex $f=(1,f')$ of $P$ such that $n_0:=\left<e-f',n\right>\in \Z$ is maximal. We have
	\begin{align*}
	\left<(1,f'),(n_0,n)\right>=\left<f',n\right> + n_0=\left<e,n\right>=\Min{P_0}{n}
	\end{align*}
	and since $\left<(0,e),(n_0,n)\right>=\left<e,n\right>=\Min{P_0}{n}$ we get that the dual lattice vector $(n_0,n)$ defines a hyperplane containing $(0,e)$ and $(1,f')$ that is also a supporting hyperplane for $P$ since we have chose $n_0$ maximal. So we get
	\begin{align*}
	\Min{P}{(n_0,n)}=\left<e,n\right>=\Min{\{0\} \times P_0}{(n_0,n)}
	\end{align*}
	and so we have $(n_0,n)\in \tilde{N_1}$.
	
	Let be $n\in N_2(P_0)$, i.e. $2\Min{P_0}{n}\in \Z\setminus 2\Z$. Then there is a vertex $e$ of $P_0$ with $\Min{P_0}{n}=\left<e,n\right>$. We now choose a vertex $f=(1,f')$ of $P$, so that $n_0:=\left<e-f',n\right>-\frac{1}{2}\in \Z$ is maximal. Then we have
	\begin{align*}
	\left<(1,f'),(n_0,n)\right>=\left<f',n\right> + n_0=\left<e,n\right>-\frac{1}{2}=\Min{P_0}{n}-\frac{1}{2}
	\end{align*}
	and so the dual lattice vector $(n_0,n)$ defines a hyperplane containing $(1,f')$ and this is also a supporting hyperplane for $P$ since we chose $n_0$ maximal. So we get
	\begin{align*}
		\Min{P}{(n_0,n)}+\frac{1}{2}=\left<e,n\right>=\Min{\{0\} \times P_0}{(n_0,n)}
	\end{align*}
	and so we have $(n_0,n)\in \tilde{N_2}$.
\end{proof}

If the middle polytope is a lattice polytope, then the situation becomes easier. This was already seen in \cite[4.3]{Boh24b}. We now understand this situation as a corollary.

\begin{coro}
Let $P\subseteq \R \times M_\R$ be a lattice polytope of lattice width $2$ with a lattice polytope $P_0\subseteq M_\R$ as middle polytope of $P$. Then $F(P)=\{0\} \times F(P_0)$.
\end{coro}
\begin{proof}
Since $P_0$ is a lattice polytope, we have $N_2(P_0)=\emptyset, N_1(P_0)=N\setminus \{0\}$ and therefore $F(P)=F_1(P)=\{0\} \times F(P_0)$.
\end{proof}

We now introduce some new notation to work directly in $M_\R$.

\begin{defi}
Let $P_0\subseteq M_\R$ be a half-integral polytope. Then we set
\begin{align*}
	\bar{F}(P_0):=\bar{F}_1(P_0)\cap \bar{F}_2(P_0)
\end{align*}
with
\begin{align*}
\bar{F}_1(P_0):=&\{x \in M_\R \mid \left<x,n\right> \geq \Min{P_0}{n}+1 \ \forall n \in N_1(P_0)\}\\
\bar{F}_2(P_0):=&\{x \in M_\R \mid \left<x,2n\right> \geq \Min{P_0}{2n}+1 \ \forall n \in N_2(P_0)\}.
\end{align*}
\end{defi}

With this notation we now have the following short version of \ref{Fine_int_by_middle_polytope}.

\begin{coro}
Let $P\subseteq \R \times M_\R$ be a lattice polytope of lattice width $2$, $P \subseteq [-1,1] \times M_\R $ and $P_0\subseteq M_\R$ the middle polytope of $P$. Then we have
\begin{align*}
	F(P)=\{0\} \times \bar{F}(P_0).
\end{align*}
\end{coro}

\begin{rem}
Note that $\bar{F}(P_0)$ is not in general the Fine interior of the half-integral polytope $P_0\subseteq M_\R$. We have that $\bar{F}(P_0)$ is the Fine interior of $P_0$ if and only if the primitive normal vectors to $F(P_0)$ are all in $N_1(P_0)$. In particular, $\bar{F}(P_0)=F(P_0)$ if $P_0$ is a lattice polytope. We also have the following inclusions
\begin{align*}
F(F(2P_0))\subseteq 2F(P_0)\subseteq 2 \bar{F}(P_0)\subseteq F(2P_0).
\end{align*}
\end{rem}

If the middle polytope completely determines the Fine interior, we can focus on some special polytopes with this middle polytope. For a rational polytope $P_0 \subseteq M_\R$ we have the pyramid over $P$ defined by $\mathrm{Pyr}(P_0):= \conv{(1,0), \{0\}\times P_0}\subseteq \R \times M_\R$ and if we translate $2\mathrm{Pyr}(P_0)$ by the lattice vector $(-1,0)$ to a subpolytope of $[-1,1] \times M_\R$, then we get $P_0$ as the middle polytope. So we have the following corollary.  

\begin{coro}\label{FineInt_by_middle_polytop}
In the situation from \ref{Fine_int_by_middle_polytope} we have
\begin{align*}
	F(P)=\{0\} \times \bar{F}(P_0) \cong  F(2\cdot \mathrm{Pyr}(P_0))
\end{align*}
Moreover, if $P_0$ is even a lattice polytope, then we also have 
\begin{align*}
	F(P)=F([-1,1] \times P_0).
\end{align*}
\end{coro}

\section{The Fine interior of a lattice $3$-polytope of lattice width $2$}

\begin{rem}
	Let $P$ be as in the theorem, with a lattice polygon $P_0\subseteq M_\R$ as middle polytope of $P$.\\
	Then $F(P)=F(P_0)\times \{0\}=\conv{\mathrm{int}(P)\cap M}$ is a lattice polygon.
\end{rem}

\subsection{The Fine interior of a lattice $3$-polytope of lattice width $2$ is a half-integral polygon}

\begin{thm}\label{Fine_int_half_integral}
Let $M\cong \Z^2$ be a lattice of rank $2$, $P\subseteq \R \times M_\R$ a lattice $3$-polytope of lattice width $2$, $P\subseteq [-1,1] \times M_\R$, and $\dim(F(P))=2$. Then $F(P)\subseteq \{0\} \times M_\R$ is a half-integral polygon.
\end{thm}
\begin{proof}
	Let be $P_0\subseteq M_\R$ the middle polygon of $P$. By \ref{Fine_int_by_middle_polytope} it is sufficient to show that $\bar{F}(P_0)$ is a half-integral polygon. To see this, we will prove that
	\begin{align*}
		\bar{F}(P_0) \subseteq \conv{x\in \bar{F}(P_0) \mid 2x \in M}.
	\end{align*}
	We can assume that we have at least two lattice points in $\bar{F}(P_0)$, since lattice 3-polytopes with at most $1$ interior lattice point cannot have a Fine interior of dimension $2$ by \cite{BKS22}. 	So $\conv{x\in \bar{F}(P_0) \mid 2x \in M}$ has at least dimension $1$ and we can look at any edge $e\preceq \conv{x\in \bar{F}(P_0) \mid 2x \in M}$. Now it is enough to see that $e$ is also an edge of $\bar{F}(P_0)$. We have two cases, either the affine hull of $e$ contains lattice points or it does not.
	
	If the affine hull of $e$ contains lattice points, then after a suitable affine unimodular transformation we can assume that the line segment $\conv{(0,0), (1/2,0)}$ is a subset of $e$ and that we have no points of $\conv{x\in \bar{F}(P_0) \mid 2x \in M}$ with positive second coordinate. If we can now show that there are no points of $P_0$ with second coordinate greater than $1$, we get that there are no points of $\bar{F}(P_0)$ with positive second coordinate and so we are done. Let $x_m=(x_{m1}, x_{m2})\in P_0$ be a half-integral point with maximal second coordinate. Then the triangle $\conv{(0,0), (1/2,0), x_m}$ contains a half-integral point with second coordinate $1/2$ by Pick's theorem. With a suitable shearing we can assume that this point is $(1/2,1/2)$ and so we get that  $1/2 \leq x_{m1} \leq x_{m2}$ by convexity of the triangle. If $x_{m2}>1$, then we have for $n\in N\setminus \{0\}$ at least one of the following inequalities 
	\begin{align*}
		\left<(0,0),n\right>\leq& \left<(1/2,1/2),n\right> \ \text{for}  \text{or}\\
		\left<(1/2,0),n\right>\leq& \left<(1/2,1/2),n\right>\ \text{or}\\
		\left<(x_m,n\right>< \left<(1,1),n\right> <& \left<(1/2,1/2),n\right>\ \text{or}\\
		\left<(x_m,n\right>< \left<(1/2,1),n\right> <& \left<(1/2,1/2),n\right>.
	\end{align*}
	These inequalities imply $\left<(1/2,1/2),n\right> \geq \Min{\bar{F}(P_0)}{n}$ and so we get $(1/2,1/2)\in \conv{x\in \bar{F}(P_0) \mid 2x \in M}$ which contradicts $x_{m2}>1$. So we must have $x_{m2}\leq 1$ and so we are done.
	
	If the affine hull of $e$ contains no lattice points then we start our argument in a similar way. After a suitable affine unimodular transformation we can assume  that the line segment $\conv{(0,-1/2), (1/2,-1/2)}$ is a subset of $e$ and we have no points of $\conv{x\in \bar{F}(P_0) \mid 2x \in M}$ with second coordinate greater than $-1/2$. If we can now proof that there are no points of $P_0$ with second coordinate greater than $0$, we are done. Let $x_m=(x_{m1}, x_{m2})\in P_0$ again be a half-integral point with maximal second coordinate. Then, by Pick's theorem, the triangle $\conv{(0,-1/2), (1/2,-1/2), x_m}$ contains a half-integral point with second coordinate $0$.  With another affine unimodular transformation we can assume that this point is $(1/2,0)$ and get so that $1/2 \leq x_{m1} \leq x_{m2}+1/2$ from the convexity of the triangle. With analogous inequalities as in the first case, we now get that $x_{m2}\leq 1/2$. So if $x_{m2}>0$, then we have $x_m=(1/2,1/2)$ or $x_m=(1,1/2)$. In the first case we get $\Min{P_0}{(1,-1)}=0$, because otherwise we get $(0,0)\in \interior{P_0}$ and therefore the contradiction $(0,0)\in \bar{F}(P_0)$. But $\Min{P_0}{(1,-1)}=0$ also leads to a contradiction, since $(0,-1/2)\in \bar{F}(P_0)$. Similary we get contradictions if $x_m=(1,1/2)$. All in all we see that $x_{m2}\leq 0$ and so we are done.
\end{proof}

The theorem also gives us also possibilities to compute $\bar{F}(P_0)$, since we only have to check which half-integral points are contained. If we know the support of the Fine interior of the lattice polygon $2P_0$, i.e. all the dual lattice vectors $n\in N$ with $\Min{2P_0}{n}+1=\Min{F(2P_0)}{n}$, we can use the following corollary. Note that the Fine interior here is just the convex hull of the interior lattice points, since $2P_0$ is a lattice polygon (\cite[2.9]{Bat17}).

\begin{coro}
Let $P$ be as in the theorem, $P_0$ its half-integral middle polygon. Then we get $2\bar{F}(P_0)$ as the convex hull of the following finite sets of lattice points
\begin{align*}
&\{x \in \partial F(2P_0)\cap M \mid \left<x,n\right> \neq \Min{2P_0}{n}+1 \forall n \in S_F(2P_0)\cap N_1(P_0) \},\\
&F(F(2P_0))\cap M.
\end{align*}
\end{coro}

From the proof of the theorem we get another corollary. We have shown there that an edge of $\bar{F}(P_0)$ without lattice point has only lattice distance $1/2$ to a supporting line which contains therefore lattice points. But this means that such an edge of $\bar{F}(P_0)$ cannot exist and we get the following corollary.

\begin{coro}\label{lpoint_on_edge}
Let $P$ be as in the theorem with $\dim F(P)=2$. Then we have for all edges $E\preceq F(P)$ that $E\cap M\neq \emptyset$.
\end{coro}

\subsection{The Fine interior vs. the convex hull of the interior lattice points}

For a lattice polygon, we know from \cite[2.9]{Bat17} that the Fine interior is just the convex hull of the interior lattice points. For $\bar{F}$ of a half-integral polygon the situation is more complicated, but there are also some connections between $\bar{F}$ and the convex hull of the interior lattice points that we want to study in this subsection.  

\begin{prop}\label{edges_of_Fineint}
Let $P_0\subseteq M_\R$ be a half-integral polygon with at least two interior lattice points, $e \preceq \conv{\interior{P_0} \cap M}$ an edge and $v \preceq \conv{\interior{P_0}\cap M}$ a vertex. Then $e$ is an edge of $\bar{F}(P_0)$ or $|e\cap M|=2$ and $v$ is a boundary lattice point of $\bar{F}(P_0)$.
\end{prop}
\begin{proof}
	First, we show that an edge $e \preceq \conv{\interior{P_0} \cap M}$ with $|e\cap M|>2$ is an edge of $\bar{F}(P_0)$. After a suitable affine unimodular transformation we can assume that $\conv{(0,0), (2,0)}\subseteq e$ and the second coordinate of all interior lattice points of $P_0$ is at most $0$. Suppose now that $e$ is not an edge of $\bar{F}(P_0)$, i.e. we have a half-integral point with positive second coordinate in $\bar{F}(P_0)$ and we cah choose $S=(s_1,s_2)\in \bar{F}(P_0)$ as a half-integral point with maximal possible second coordinate $s_2$. By a suitable shearing we can assume that $s_2 \geq s_1>0$. We conclude directly that $s_2<\frac{3}{2}$,  since otherwise $\conv{(0,0), (2,0), (s_1,s_2)}\subseteq \interior{P_0}$ contains the lattice point $(1,1)\in M$, which contradicts our choice of coordinates. It remains to consider the cases $S=(1/2,1)$ and $S=(1/2,1/2)$. Since $s_2$ was chosen maximal and every edge of $\bar{F}(P_0)$ contains a lattice point by \ref{lpoint_on_edge}, $S$ must be a vertex of $\bar{F}(P_0)$. Let $e_{S,l}, e_{S,r} \preceq \bar{F}(P_0)$ be the edges of $\bar{F}(P_0)$ containing $S$, where the indices $l$ and $r$ indicate that the first coordinates of the points of $e_{S_l}$ are not greater than those of $e_{S_r}$ and so $e_{S_l}$ is on the left. Take the parallel lines to the affine hull of $e_{S_l}$ through $(0,1)$ and to the affine hull of $e_{S_2}$ through $(1,1)$. Then $P_0$ must be below both of these parallel. This implies that all points of $P_0$ have second coordinate of at most $1$, which contradicts the fact that $S\in \bar{F}(P_0)$.
	
	We now show that every vertex of $\conv{\interior{P_0}\cap M}$ is a boundary lattice point of $\bar{F}(P_0)$. From the first part of the proof we already know that this is true for vertices on edges of $\bar{F}(P_0)$ with at least $3$ lattice points. So we only need to consider edges of $\conv{\interior{P_0}\cap M}$ with exactly two lattice points. After a suitable affine unimodular transformation we can assume that this edge is $\conv{(0,0), (1,0)}$ and that the second coordinate of all interior lattice points of $P_0$ to be at most $0$. Furthermore, we can still assume that $S=(s_1,s_2)\in \bar{F}(P_0)$ is a half-integral point with the maximum possible second coordinate $s_2$ and $s_2 \geq s_1>0$. We now get $s_1=\frac{1}{2}$, since otherwise $\conv{(0,0), (1,0), (s_1,s_2)}\subseteq \interior{P_0}$ contains the lattice point $(1,1)\in M$, which contradicts our choice of coordinates. By symmetry it is now sufficient to show that the line segment $\conv{(0,0), S}$ is a subset of an edge of $\bar{F}(P_0)$. As in the first part of the proof, we consider parallel lines to the affine hull of $e_{S_l}$ through $(0,1)$ and to the affine hull of $e_{S_2}$ through $(1,1)$. If $\conv{(0,0), S}$ is not a subset of an edge of $\bar{F}(P_0)$, then we get that all points of $P_0$ have a second coordinate of at most $s_2+\frac{1}{2}$. If $s_2\in \frac{1}{2}\Z \setminus \Z$, then this contradicts the fact that $S\in \bar{F}(P_0)$.If $s_2\in \Z$, then the affine hull of $(0,1)$ and $(1/2,s_2+1/2)$ must be a supporting line of $P_0$. But this also contradicts  $S\in \bar{F}(P_0)$ and so we are done.  
\end{proof}

Since we have seen, that edges of $\conv{\interior{P_0} \cap M}$ with $3$ or more lattice points are always edges of $\bar{F}(P_0)$, we get the following corollary

\begin{coro}
Let $P_0\subseteq M_\R$ be a half-integral polygon with at least three interior lattice points, which are collinear. Then $\dim \bar{F}(P_0)=1$ and we have in particular that $P_0$ is affine unimodular equivalent to a polygon in $\R \times [-1,1]$.
\end{coro}

We now focus on the edges of $\conv{\interior{P_0} \cap M}$ which are edges of $\bar{F}(P_0)$. For them we have the following proposition.

\begin{prop}
Let $P_0\subseteq M_\R$ be a half-integral polygon with at least two interior lattice points, $e \preceq \conv{\interior{P_0} \cap M}$ an edge which is not an edge of $\bar{F}(P_0)$. Then there is only one vertex $v$ of $\bar{F}(P)$, which is on the other side of $e$ than $\interior{P_0} \cap M$. Moreover, if the lattice distance of $v$ from $e$ is $h$, then $P_0$ has at least $2(h-1)$ interior lattice points, and the triangle $\conv{e,v}$ is unimodular affine equivalent to the half-integral triangle $\conv{(0,0),(1,0),(1/2,h)}$. 
\end{prop}
\begin{proof}
	We have already seen in the proof of the second part of \ref{edges_of_Fineint}, that there is only one vertex $v$ of $\bar{F}(P)$, which is on the other side of $e$ than $\interior{P_0} \cap M$ and that the triangle $\conv{e,v}$ is unimodular affine equivalent to the half-integral triangle $\conv{(0,0),(1,0),(1/2,h)}$. So we can assume that the triangle is $\conv{(0,0),(1,0),(1/2,h)}$ and we can shift the affine hulls of $(0,0), (1/2,h)$ and $(1,0), (1/2,h)$ by one integral step outside and bound with them an area which contains $P_0$. In particular, since $(0,0), (1,0)$ are interior lattice points of $P_0$ there must a points $(x_1,x_2)\in P_0$ with $x_1\leq -1/2, x_2\leq -h+1$ as well as a point with $x_1\geq 3/2, x_2\leq -h+1$. Therefore all the lattice points in 
	\begin{align*}
		\conv{(0,0), (1,0), (0,-h-2), (1,h-2)}
	\end{align*}
	are interior lattice point of $P_0$ and we therefore have at least $(2(h-1))$ interior lattice points of $P_0$.
\end{proof}

\begin{coro}
The lattice height $h$ of the triangles on edges of $\conv{\interior{P_0} \cap M}$ with two lattice points from the last proposition is bounded by $\frac{1}{2}(|P_0\cap M|+2)$.
\end{coro}

Let us end this subsection with summerizing some propeties of $\bar{F}(P_0)$ we have proven so far.

\begin{thm}\label{hats}
	Let $P_0\subseteq M_\R$ be a half-integral polygon with $\dim(\bar{F}(P_0))=2$. Then $\bar{F}(P_0)$ is a half-integral polygon, every edge of $\bar{F}(P_0)$ contains lattice points and we get $\bar{F}(P_0)$ as union of $\conv{\interior{P_0} \cap M}$ with some triangles, whereby the base of each of these triangle is an edge of $\conv{\interior{P_0} \cap M}$ with two lattice point and the third vertex of the triangle has a lattice height of at most $\frac{1}{2}(|P_0\cap M|+2)$ over this edge. Moreover, every vertex of $\conv{\interior{P_0} \cap M}$ is a boundary lattice point of $\bar{F}(P_0)$.
\end{thm}

\subsection{The maximal half-integral polygon to a given Fine interior}

In this subsection we want to describe the unique inclusion maximal lattice polytope to a given Fine interior and the unique inclusion maximal half-integral polygon for given $\bar{F}$. This leads also to a test, if a given polytope can occure as a Fine interior.

We start by moving out the facets of a polytope by lattice distance $1$.

\begin{defi}
Let $P\subseteq M_\R$ be a rational polytope. Then we define
\begin{align*}
P^{(-1)}:=\{x \in M_\R \mid \left<x,n\right> \geq \Min{P}{n}-1 \ \forall n \in \normalfan{P}[1]\}.
\end{align*}
\end{defi}

\begin{rem}
Moving out the facets is partially an inverse operation of computing the Fine interior. We get directly from the definition the following relations
\begin{align*}
	F(\conv{P^{(-1)}\cap M)}\subseteq F(P^{(-1)})\subseteq P.
\end{align*}
The opposite inclusions are not correct in general but they hold for Fine interiors as we see in the next lemma.
\end{rem}

\begin{lemma}
A rational $d$-polytope $Q\subseteq M_\R$ is a Fine interior of a rational polytope if and only if $Q\subseteq F(Q^{(-1)})$. Moreover, $Q$ is a Fine interior of a lattice polytope if and only if $Q\subseteq F(\conv{Q^{(-1)}\cap M)}$.
\end{lemma}
\begin{proof}
	If the inclusions hold then $Q$ is a Fine interior since the opposite inclusions hold in general. If $Q$ is a Fine interior, then there is a $d$-polytope $Q'\subseteq M_\R$ with $F(Q')=Q$. We get $Q'\subseteq Q^{(-1)}$ and so by the monotony of the Fine interior $Q=F(Q')\subseteq F(Q^{(-1)})$.
\end{proof}


The situation becomes especially easy for lattice polygons. We have by \cite[Lemma 9]{HS09} the following proposition. 

\begin{prop}
A lattice polygon $Q$ is a Fine interior of a lattice polygon if and only if $Q^{(-1)}$ is a lattice polygon.
\end{prop}

\begin{rem}
All lattice polygons with up to 112 lattice points which are a Fine interior a lattice polygon were classified for calculations in \cite{BS24a}. The data are available on zenodo \cite{BS24b}.
\end{rem}

We now want similar results as for the Fine interior also for $\bar{F}$.

\begin{defi}
Let $P_0\subseteq M$ be a half-integral polygon with $\dim \bar{F}(P_0)=2$. Then we call $P_0$ a maximal half-integral polygon for given $\bar{F}$ if and only if 
\begin{align*}
P_0=\conv{(\bar{F}(P_0))^{(-1)}\cap M/2}.
\end{align*} 
\end{defi}

\begin{rem}
The maximal half-integral polygon for given $\bar{F}$ is unique by definition and contains all other half-integral polygons with the same $\bar{F}$. If a half-integral polygon is maximal with respect to inclusion among all half-integral polygons with the same number of interior lattice points, it is in particular maximal for given $\bar{F}$. Thus we get all inclusion maximal half-integral polygons as a subset of the polygons maximal for given $\bar{F}$. This was used to classify half-integral polygons in \cite{BS24a}.
\end{rem}


We get now analogue to above a test for Fine interiors of dimension $2$ of lattice $3$-polytopes.

\begin{coro}\label{FineInt_test}
A half-integral polygon $Q\subseteq M_\R$ is unimodular equivalent to the Fine interior of a lattice $3$-polytope if and only if 
\begin{align*}
Q=\bar{F}(\conv{Q^{(-1)}\cap M/2}).
\end{align*} 
\end{coro}

\section{Classification of Fine interiors}

The previous section gives us the following algorithm to classify the two dimensional Fine interiors of lattice $3$ polytopes with given number of lattice points.

\begin{enumerate}
	\item Take all lattice polygons with the given number of lattice points, e.g. from the dataset \cite{BS24b}.
	\item For each lattice polygon with the given number of lattice points consider all possibilities to put triangles on the edges with $2$ lattice points, which are permitted by \ref{hats}. We get a finite number of lattice polygons with hats, since the height of the hats is bounded and the lattice polygon with hats needs to be convex.
	\item Check for each lattice polygon with hats from the last step with \ref{FineInt_test}, if it is indeed a Fine interior or not and discard the polygons which are not a Fine interior.
	\item Eliminate affine unimodular equivalent Fine interior, e. g. by use of a normal form as described in \cite[section 2.1.]{BS24a}.
\end{enumerate}

This algorithm was used for calculations with up to $40$ lattice points and we got the following result.

\begin{thm}\label{classification}
	There are exactly 24 324 158 half-integral polygons with at most 40
	lattice points, which are Fine interiors of a lattice 3-polytope. The data for the vertices of these polygons are available on \cite{Boh24a}. The examples with $2$ or $3$ lattice points can be seen in figure \ref{2points} and \ref{3points}.
\end{thm}

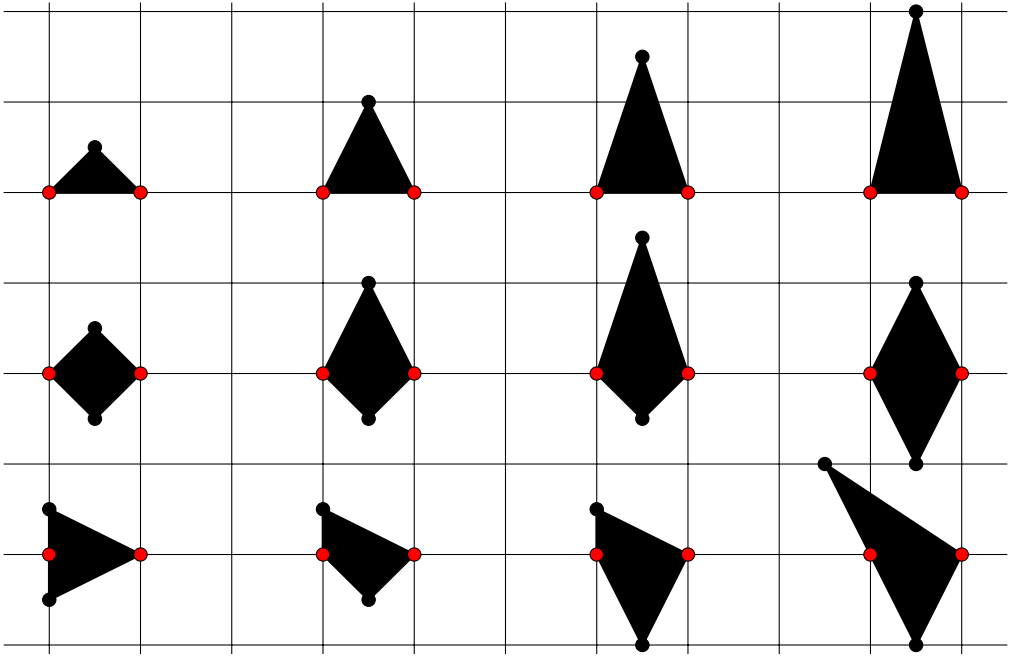
\begin{figure}
\begin{tikzpicture}[x=0.6cm,y=0.6cm]
			\draw[step=2.0,black,thin,xshift=0cm,yshift=0cm] (-11,12.2) grid (11,-2.2);
			
			\fill[line width=2pt,color=black,fill=black,fill opacity=0.3] (-10,8) -- (-9,9) -- (-8,8) -- cycle;
			\draw [line width=1pt,color=black] (-9,9)-- (-10,8);
			\draw [line width=1pt,color=black] (-9,9)-- (-8,8);
			\draw [line width=1pt,color=black] (-8,8)-- (-10,8);
			
			\draw [fill=black] (-9,9) circle (2.5pt);
			\draw [fill=black] (-10,8) circle (2.5pt);
			\draw [fill=black] (-8,8) circle (2.5pt);

			\fill[line width=2pt,color=black,fill=black,fill opacity=0.3] (-4,8) -- (-3,10) -- (-2,8) -- cycle;
			\draw [line width=1pt,color=black] (-3,10)-- (-4,8);
			\draw [line width=1pt,color=black] (-3,10)-- (-2,8);
			\draw [line width=1pt,color=black] (-2,8)-- (-4,8);
			
			\draw [fill=black] (-3,10) circle (2.5pt);
			\draw [fill=black] (-4,8) circle (2.5pt);
			\draw [fill=black] (-2,8) circle (2.5pt);

			\fill[line width=2pt,color=black,fill=black,fill opacity=0.3] (2,8) -- (3,11) -- (4,8) -- cycle;
			\draw [line width=1pt,color=black] (3,11)-- (2,8);
			\draw [line width=1pt,color=black] (3,11)-- (4,8);
			\draw [line width=1pt,color=black] (4,8)-- (2,8);
			
			\draw [fill=black] (3,11) circle (2.5pt);
			\draw [fill=black] (2,8) circle (2.5pt);
			\draw [fill=black] (4,8) circle (2.5pt);

			\fill[line width=2pt,color=black,fill=black,fill opacity=0.3] (8,8) -- (9,12) -- (10,8) -- cycle;
			\draw [line width=1pt,color=black] (9,12)-- (8,8);
			\draw [line width=1pt,color=black] (9,12)-- (10,8);
			\draw [line width=1pt,color=black] (10,8)-- (8,8);
			
			\draw [fill=black] (9,12) circle (2.5pt);
			\draw [fill=black] (8,8) circle (2.5pt);
			\draw [fill=black] (10,8) circle (2.5pt);

			\fill[line width=2pt,color=black,fill=black,fill opacity=0.3] (-10,4) -- (-9,5) -- (-8,4) -- (-9,3) -- cycle;
			\draw [line width=1pt,color=black] (-10,4)-- (-9,5);
			\draw [line width=1pt,color=black] (-9,5)-- (-8,4);
			\draw [line width=1pt,color=black] (-8,4)-- (-9,3);
			\draw [line width=1pt,color=black] (-10,4)-- (-9,3);
			
			\draw [fill=black] (-9,5) circle (2.5pt);
			\draw [fill=black] (-9,3) circle (2.5pt);
			\draw [fill=black] (-10,4) circle (2.5pt);
			\draw [fill=black] (-8,4) circle (2.5pt);

			\fill[line width=2pt,color=black,fill=black,fill opacity=0.3] (-4,4) -- (-3,6) -- (-2,4) -- (-3,3)  -- cycle;
			\draw [line width=1pt,color=black] (-4,4)-- (-3,6);
			\draw [line width=1pt,color=black] (-3,6)-- (-2,4);
			\draw [line width=1pt,color=black] (-2,4)-- (-3,3);
			\draw [line width=1pt,color=black] (-3,3)-- (-4,4);
			
			\draw [fill=black] (-3,6) circle (2.5pt);
			\draw [fill=black] (-3,3) circle (2.5pt);
			\draw [fill=black] (-4,4) circle (2.5pt);
			\draw [fill=black] (-2,4) circle (2.5pt);

			\fill[line width=2pt,color=black,fill=black,fill opacity=0.3] (2,4) -- (3,7) -- (4,4) -- (3,3) -- cycle;
			\draw [line width=1pt,color=black] (2,4)-- (3,7);
			\draw [line width=1pt,color=black] (3,7)-- (4,4);
			\draw [line width=1pt,color=black] (4,4)-- (3,3);
			\draw [line width=1pt,color=black] (3,3)-- (2,4);
			
			\draw [fill=black] (3,7) circle (2.5pt);
			\draw [fill=black] (3,3) circle (2.5pt);
			\draw [fill=black] (2,4) circle (2.5pt);
			\draw [fill=black] (4,4) circle (2.5pt);

			\fill[line width=2pt,color=black,fill=black,fill opacity=0.3] (8,4) -- (9,6) -- (10,4) -- (9,2) -- cycle;
			\draw [line width=1pt,color=black] (8,4)-- (9,6);
			\draw [line width=1pt,color=black] (9,6)-- (10,4);
			\draw [line width=1pt,color=black] (10,4)-- (9,2);
			\draw [line width=1pt,color=black] (8,4)-- (9,2);
			
			\draw [fill=black] (9,6) circle (2.5pt);
			\draw [fill=black] (9,2) circle (2.5pt);
			\draw [fill=black] (8,4) circle (2.5pt);
			\draw [fill=black] (10,4) circle (2.5pt);
			
			\fill[line width=2pt,color=black,fill=black,fill opacity=0.3] (-10,0) -- (-10,1) -- (-8,0) -- (-10,-1) -- cycle;
			\draw [line width=1pt,color=black] (-10,0)-- (-10,1);
			\draw [line width=1pt,color=black] (-10,1)-- (-8,0);
			\draw [line width=1pt,color=black] (-8,0)-- (-10,-1);
			\draw [line width=1pt,color=black] (-10,-1)-- (-10,0);
			
			\draw [fill=black] (-10,1) circle (2.5pt);
			\draw [fill=black] (-10,-1) circle (2.5pt);
			\draw [fill=black] (-10,0) circle (2.5pt);
			\draw [fill=black] (-8,0) circle (2.5pt);

			\fill[line width=2pt,color=black,fill=black,fill opacity=0.3] (-4,0) -- (-4,1) -- (-2,0) -- (-3,-1)  -- cycle;
			\draw [line width=1pt,color=black] (-4,0)-- (-4,1);
			\draw [line width=1pt,color=black] (-4,1)-- (-2,0);
			\draw [line width=1pt,color=black] (-2,0)-- (-3,-1);
			\draw [line width=1pt,color=black] (-3,-1)-- (-4,0);
			
			\draw [fill=black] (-4,1) circle (2.5pt);
			\draw [fill=black] (-3,-1) circle (2.5pt);
			\draw [fill=black] (-4,0) circle (2.5pt);
			\draw [fill=black] (-2,0) circle (2.5pt);

			\fill[line width=2pt,color=black,fill=black,fill opacity=0.3] (2,0) -- (2,1) -- (4,0) -- (3,-2) -- cycle;
			\draw [line width=1pt,color=black] (2,0)-- (2,1);
			\draw [line width=1pt,color=black] (2,1)-- (4,0);
			\draw [line width=1pt,color=black] (4,0)-- (3,-2);
			\draw [line width=1pt,color=black] (3,-2)-- (2,0);
			
			\draw [fill=black] (2,1) circle (2.5pt);
			\draw [fill=black] (3,-2) circle (2.5pt);
			\draw [fill=black] (2,0) circle (2.5pt);
			\draw [fill=black] (4,0) circle (2.5pt);

			\fill[line width=2pt,color=black,fill=black,fill opacity=0.3] (9,-2) -- (10,0) -- (7,2) -- cycle;
			\draw [line width=1pt,color=black] (8,0)-- (9,-2);
			\draw [line width=1pt,color=black] (9,-2)-- (10,0);
			\draw [line width=1pt,color=black] (10,0)-- (7,2);
			\draw [line width=1pt,color=black] (7,2)-- (8,0);
			
			\draw [fill=black] (7,2) circle (2.5pt);
			\draw [fill=black] (9,-2) circle (2.5pt);
			\draw [fill=black] (8,0) circle (2.5pt);
			\draw [fill=black] (10,0) circle (2.5pt);

\end{tikzpicture}

		\caption{All Fine interiors of dimension $2$ for lattice $3$-polytopes with $2$ interior lattice points}\label{2points}
\end{figure}

\begin{figure}
	\begin{tikzpicture}[x=0.6cm,y=0.6cm]
		\draw[step=2.0,black,thin,xshift=0.6cm,yshift=0cm] (-13.9,10.2) grid (11.9,-7.2);
		
		\fill[line width=2pt,color=black,fill=black,fill opacity=0.3] (-11,8) -- (-9,8) -- (-11,10) -- cycle;
		\draw [line width=1pt,color=black] (-11,8)-- (-9,8);
		\draw [line width=1pt,color=black] (-9,8)-- (-11,10);
		\draw [line width=1pt,color=black] (-11,10)-- (-11,8);
		
		\draw [fill=black] (-11,8) circle (2.5pt);
		\draw [fill=black] (-9,8) circle (2.5pt);
		\draw [fill=black] (-11,10) circle (2.5pt);

		\fill[line width=2pt,color=black,fill=black,fill opacity=0.3] (-7,7)  -- (-5,8) -- (-7,10) -- cycle;
		\draw [line width=1pt,color=black] (-7,7)-- (-5,8);
		\draw [line width=1pt,color=black] (-5,8)-- (-7,10);
		\draw [line width=1pt,color=black] (-7,10)-- (-7,7);
		
		\draw [fill=black] (-7,7) circle (2.5pt);
		\draw [fill=black] (-7,8) circle (2.5pt);
		\draw [fill=black] (-5,8) circle (2.5pt);
		\draw [fill=black] (-7,10) circle (2.5pt);

		\fill[line width=2pt,color=black,fill=black,fill opacity=0.3] (-3,8) -- (-2,7) -- (-1,8) -- (-3,10) -- cycle;
		\draw [line width=1pt,color=black] (-3,8)-- (-2,7);
		\draw [line width=1pt,color=black] (-2,7)-- (-1,8);
		\draw [line width=1pt,color=black] (-1,8)-- (-3,10);
		\draw [line width=1pt,color=black] (-3,10)-- (-3,8);
		
		\draw [fill=black] (-2,7) circle (2.5pt);
		\draw [fill=black] (-3,8) circle (2.5pt);
		\draw [fill=black] (-1,8) circle (2.5pt);
		\draw [fill=black] (-3,10) circle (2.5pt);

		\fill[line width=2pt,color=black,fill=black,fill opacity=0.3] (1,8) -- (2,6) -- (3,8) -- (1,10) -- cycle;
		\draw [line width=1pt,color=black] (1,8)-- (2,6);
		\draw [line width=1pt,color=black] (2,6)-- (3,8);
		\draw [line width=1pt,color=black] (3,8)-- (1,10);
		\draw [line width=1pt,color=black] (1,10)-- (1,8);
		
		\draw [fill=black] (2,6) circle (2.5pt);
		\draw [fill=black] (1,8) circle (2.5pt);
		\draw [fill=black] (3,8) circle (2.5pt);
		\draw [fill=black] (1,10) circle (2.5pt);
		
		\fill[line width=2pt,color=black,fill=black,fill opacity=0.3] (5,8) -- (6,5) -- (7,8) -- (5,10) -- cycle;
		\draw [line width=1pt,color=black] (5,8)-- (6,5);
		\draw [line width=1pt,color=black] (6,5)-- (7,8);
		\draw [line width=1pt,color=black] (7,8)-- (5,10);
		\draw [line width=1pt,color=black] (5,8)-- (5,10);
		
		\draw [fill=black] (6,5) circle (2.5pt);
		\draw [fill=black] (5,8) circle (2.5pt);
		\draw [fill=black] (7,8) circle (2.5pt);
		\draw [fill=black] (5,10) circle (2.5pt);

		\fill[line width=2pt,color=black,fill=black,fill opacity=0.3] (-10,1) -- (-9,2) -- (-11,4) -- (-12,3) --  cycle;
		\draw [line width=1pt,color=black] (-12,3)-- (-10,1);
		\draw [line width=1pt,color=black] (-10,1)--  (-9,2);
		\draw [line width=1pt,color=black] (-9,2)-- (-11,4);
		\draw [line width=1pt,color=black] (-11,4)-- (-12,3);
		
		\draw [fill=black] (-10,1) circle (2.5pt);
		\draw [fill=black] (-12,3) circle (2.5pt);
		\draw [fill=black] (-11,2) circle (2.5pt);
		\draw [fill=black] (-9,2) circle (2.5pt);
		\draw [fill=black] (-11,4) circle (2.5pt);

		\fill[line width=2pt,color=black,fill=black,fill opacity=0.3] (-8,4)  -- (-7,2) -- (-5,1) -- (-5,2) -- (-7,4) --  cycle;
		\draw [line width=1pt,color=black] (-8,4)-- (-7,2);
		\draw [line width=1pt,color=black] (-7,2)-- (-5,1);
		\draw [line width=1pt,color=black] (-5,1)-- (-5,2);
		\draw [line width=1pt,color=black] (-5,2)-- (-7,4);
		\draw [line width=1pt,color=black] (-7,4)-- (-8,4);
		
		\draw [fill=black] (-8,4) circle (2.5pt);
		\draw [fill=black] (-5,1) circle (2.5pt);
		\draw [fill=black] (-7,2) circle (2.5pt);
		\draw [fill=black] (-5,2) circle (2.5pt);
		\draw [fill=black] (-7,4) circle (2.5pt);

		\fill[line width=2pt,color=black,fill=black,fill opacity=0.3] (-4,4) -- (-3,2) -- (-2,1) -- (-1,2) -- (-3,4) -- cycle;
		\draw [line width=1pt,color=black] (-4,4)-- (-3,2);
		\draw [line width=1pt,color=black] (-3,2)-- (-2,1);
		\draw [line width=1pt,color=black] (-2,1)-- (-1,2);
		\draw [line width=1pt,color=black] (-1,2)-- (-3,4);
		\draw [line width=1pt,color=black] (-3,4)-- (-4,4);
		
		\draw [fill=black] (-4,4) circle (2.5pt);
		\draw [fill=black] (-2,1) circle (2.5pt);
		\draw [fill=black] (-3,2) circle (2.5pt);
		\draw [fill=black] (-1,2) circle (2.5pt);
		\draw [fill=black] (-3,4) circle (2.5pt);

		\fill[line width=2pt,color=black,fill=black,fill opacity=0.3] (3,2) -- (0,5) -- (1,2) -- (2,1) -- cycle;
		\draw [line width=1pt,color=black] (3,2)-- (0,5);
		\draw [line width=1pt,color=black] (0,5)-- (1,2);
		\draw [line width=1pt,color=black] (1,2)-- (2,1);
		\draw [line width=1pt,color=black] (2,1)-- (3,2);
		
		\draw [fill=black] (2,1) circle (2.5pt);
		\draw [fill=black] (0,5) circle (2.5pt);
		\draw [fill=black] (1,2) circle (2.5pt);
		\draw [fill=black] (3,2) circle (2.5pt);
		\draw [fill=black] (1,4) circle (2.5pt);
		
		\fill[line width=2pt,color=black,fill=black,fill opacity=0.3] (4,5) -- (5,2) -- (7,1) -- (7,2) -- cycle;
		\draw [line width=1pt,color=black] (4,5)-- (5,2);
		\draw [line width=1pt,color=black] (5,2)-- (7,1);
		\draw [line width=1pt,color=black] (7,1)-- (7,2);
		\draw [line width=1pt,color=black] (7,2)-- (4,5);
		
		\draw [fill=black] (4,5) circle (2.5pt);
		\draw [fill=black] (7,1) circle (2.5pt);
		\draw [fill=black] (5,2) circle (2.5pt);
		\draw [fill=black] (7,2) circle (2.5pt);
		\draw [fill=black] (5,4) circle (2.5pt);
		
				\fill[line width=2pt,color=black,fill=black,fill opacity=0.3] (8,4) -- (9,2) -- (12,0) -- (11,2) -- (9,4) --cycle;
		\draw [line width=1pt,color=black] (8,4)-- (9,2);
		\draw [line width=1pt,color=black] (9,2)-- (12,0);
		\draw [line width=1pt,color=black] (12,0)-- (11,2);
		\draw [line width=1pt,color=black] (11,2)-- (9,4);
		\draw [line width=1pt,color=black] (9,4)-- (8,4);
		
		\draw [fill=black] (8,4) circle (2.5pt);
		\draw [fill=black] (12,0) circle (2.5pt);
		\draw [fill=black] (9,2) circle (2.5pt);
		\draw [fill=black] (11,2) circle (2.5pt);
		\draw [fill=black] (9,4) circle (2.5pt);

		\fill[line width=2pt,color=black,fill=black,fill opacity=0.3] (-12,-2) -- (-10,-6) -- (-9,-4) -- (-11,-2) --  cycle;
		\draw [line width=1pt,color=black] (-12,-2)-- (-10,-6);
		\draw [line width=1pt,color=black] (-10,-6)--  (-9,-4);
		\draw [line width=1pt,color=black] (-9,-4)-- (-11,-2);
		\draw [line width=1pt,color=black] (-11,-2)-- (-12,-2);
		
		\draw [fill=black] (-12,-2) circle (2.5pt);
		\draw [fill=black] (-10,-6) circle (2.5pt);
		\draw [fill=black] (-11,-4) circle (2.5pt);
		\draw [fill=black] (-9,-4) circle (2.5pt);
		\draw [fill=black] (-11,-2) circle (2.5pt);

		\fill[line width=2pt,color=black,fill=black,fill opacity=0.3] (-8,-1)  -- (-7,-4) -- (-6,-6) -- (-5,-4) --  cycle;
		\draw [line width=1pt,color=black] (-8,-1)-- (-7,-4);
		\draw [line width=1pt,color=black] (-7,-4)-- (-6,-6);
		\draw [line width=1pt,color=black] (-6,-6)-- (-5,-4);
		\draw [line width=1pt,color=black] (-5,-4)-- (-8,-1);
		
		\draw [fill=black] (-8,-1) circle (2.5pt);
		\draw [fill=black] (-6,-6) circle (2.5pt);
		\draw [fill=black] (-7,-4) circle (2.5pt);
		\draw [fill=black] (-5,-4) circle (2.5pt);
		\draw [fill=black] (-7,-2) circle (2.5pt);

		\fill[line width=2pt,color=black,fill=black,fill opacity=0.3] (-4,-1) -- (-2,-7) -- (-1,-4) -- cycle;
		\draw [line width=1pt,color=black] (-4,-1)-- (-2,-7);
		\draw [line width=1pt,color=black] (-2,-7)-- (-1,-4);
		\draw [line width=1pt,color=black] (-1,-4)-- (-4,-1);
		
		\draw [fill=black] (-4,-1) circle (2.5pt);
		\draw [fill=black] (-2,-7) circle (2.5pt);
		\draw [fill=black] (-3,-4) circle (2.5pt);
		\draw [fill=black] (-1,-4) circle (2.5pt);
		\draw [fill=black] (-3,-2) circle (2.5pt);

		\fill[line width=2pt,color=black,fill=black,fill opacity=0.3] (0,-3) -- (2,-5) -- (3,-4) -- (3,-3) -- (1,-2) -- cycle;
		\draw [line width=1pt,color=black] (0,-3)-- (2,-5);
		\draw [line width=1pt,color=black] (2,-5)-- (3,-4);
		\draw [line width=1pt,color=black] (3,-4)-- (3,-3);
		\draw [line width=1pt,color=black] (3,-3)-- (1,-2);
		\draw [line width=1pt,color=black] (1,-2)-- (0,-3);
		
		\draw [fill=black] (3,-3) circle (2.5pt);
		\draw [fill=black] (0,-3) circle (2.5pt);
		\draw [fill=black] (2,-5) circle (2.5pt);
		\draw [fill=black] (1,-4) circle (2.5pt);
		\draw [fill=black] (3,-4) circle (2.5pt);
		\draw [fill=black] (1,-2) circle (2.5pt);
		
		\fill[line width=2pt,color=black,fill=black,fill opacity=0.3] (4,-3) -- (5,-4) -- (7,-5) -- (7,-4) -- (6,-2) -- (5,-2) -- cycle;
		\draw [line width=1pt,color=black] (4,-3)-- (5,-4);
		\draw [line width=1pt,color=black] (5,-4)-- (7,-5);
		\draw [line width=1pt,color=black] (7,-5)-- (7,-4);
		\draw [line width=1pt,color=black] (7,-4)-- (6,-2);
		\draw [line width=1pt,color=black] (6,-2)-- (5,-2);
		\draw [line width=1pt,color=black] (5,-2)-- (4,-3);
		
		\draw [fill=black] (6,-2) circle (2.5pt);
		\draw [fill=black] (4,-3) circle (2.5pt);
		\draw [fill=black] (7,-5) circle (2.5pt);
		\draw [fill=black] (5,-4) circle (2.5pt);
		\draw [fill=black] (7,-4) circle (2.5pt);
		\draw [fill=black] (5,-2) circle (2.5pt);
		
		\fill[line width=2pt,color=black,fill=black,fill opacity=0.3] (7,-3) -- (11,-5) -- (11,-3) -- (9,-2) --cycle;
		\draw [line width=1pt,color=black] (7,-3)-- (11,-5);
		\draw [line width=1pt,color=black] (11,-5)-- (11,-3);
		\draw [line width=1pt,color=black] (11,-3)-- (9,-2);
		\draw [line width=1pt,color=black] (9,-2)-- (7,-3);
		
		\draw [fill=black] (11,-3) circle (2.5pt);
		\draw [fill=black] (11,-5) circle (2.5pt);
		\draw [fill=black] (7,-3) circle (2.5pt);
		\draw [fill=black] (9,-4) circle (2.5pt);
		\draw [fill=black] (11,-4) circle (2.5pt);
		\draw [fill=black] (9,-2) circle (2.5pt);

	\end{tikzpicture}
		\caption{All Fine interiors of dimension $2$ for lattice 3-polytopes with $3$ interior lattice points}\label{3points}
\end{figure}
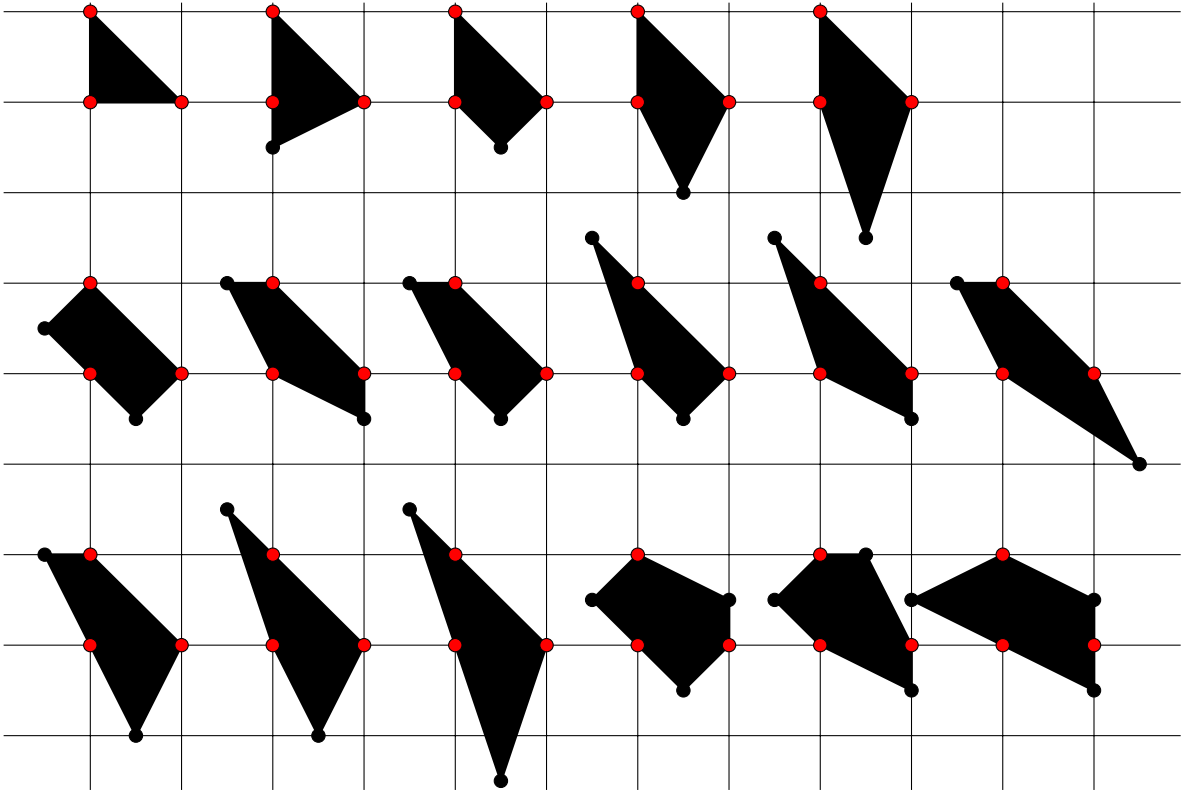

\section{Geography of minimal surfaces of general type arising from lattice $3$-polytopes of lattice width $2$}

We can also visualize invariants of the minimal surfaces which correspond the lattice $3$-polytopes with Fine interior of dimension $2$.

\begin{thm}(\cite[9.4]{Bat23}, \cite[5.1]{Gie22}) Let $\hat{Z}$ be a minimal surface coming from a nondegenerated toric hypersurface with Newton polytope $P$ of dimension $d$ with $\dim(F(P))=d-1$. Then
\begin{align*}
(K_{\hat{Z}})^{d-1}=2\mathrm{Vol}_{d-1}(F(P))
\end{align*}
\end{thm}

\begin{coro}
If $\dim (F(P))\geq d-1$, then $2 \mathrm{Vol}_{\dim F(P)}(F(P))\in \Z$.
\end{coro}
\begin{proof}
If $\dim (F(P))=d-1$, then $2\mathrm{Vol}_{d-1}(F(P))=(K_{\hat{Z}})^{d-1}\in \Z$. If $\dim (F(P))=d$, then we have for $Q:=P \times [-1,1]$, that $F(Q)=F(P)$ and $2\mathrm{Vol}_{d}(F(P))=2\mathrm{Vol}_{d}(F(Q))=(K_{\hat{Z}})^{d}\in \Z$.
\end{proof}

If $d=3$, then we get for the Chern numbers the following corollary.

\begin{coro}\cite[10.1]{Bat23}
\begin{align*}
	c_1^2(\hat{Z})=&2\mathrm{Vol}_{2}(F(P))\\
	c_2(\hat{Z})=&12(1+\chi)-c_1^2=12(1+|F(P)\cap M|)-c_1^2
\end{align*}
\end{coro}

Remember that the Chern numbers of surfaces of general type are restricted by general theory. We have the following theorem

\begin{thm}\cite[VII. (1.1)]{BHPV04}
$c_1^2, c_2\in \Z_{\geq 0}$, $c_1^2+c_2=12\chi \equiv 0 \mod 12$ and
\begin{align*}
c_1^2\leq& 3c_2  &&\quad (\text{Bogomolov–Miyaoka–Yau inequality})\\
c_1^2\geq 2\chi-4=& \frac{c_2-36}{5} &&\quad (\text{Noether inequality})
\end{align*}
\end{thm}

For surfaces on the Noether line we have the following corollary.

\begin{coro}
$P$ defines a surface with Chern numbers on the Noether line, i.e. with $c_1^2=2\chi-4$, if and only if $F(P)$ is a hollow polygon, i.e. it has no interior lattice points. Moreover, for every point on the Noether line exists a suitable hollow $F(P)$.
\end{coro}

\begin{proof}
$c_1^2\geq 2\chi-4$ gives for the Fine interior $2\mathrm{Vol}_2(F(P))\geq 2|F(P)\cap M|-4$. By Pick's formula $2\mathrm{Vol}_2(F(P))\geq 4|F(P)\cap M|-2|\partial F(P) \cap M|-4\geq 2| F(P) \cap M|-4$ with equality if and only if $F(P)$ is a hollow polygon. 
\end{proof}



If the Fine interior is a lattice polygon, we get $c_1^2, c_2 \in 2\Z_{\geq 0}$ and the following additional restrictions from the theory of lattice polygons:

\begin{thm}[\cite{Sco76}]
$|\partial P \cap M|\leq 2|\mathrm{int}(P)\cap M|+6$ or $P\cong 3\Delta_2$
\end{thm}

\begin{coro}
\begin{align*}
c_1^2=2\mathrm{Vol}_2(F(P))\geq \frac{8}{3}|F(P)\cap M|-8=\frac{2}{7}(c_2-48)
\end{align*}
\end{coro}
\begin{proof}
We have $3|\partial P \cap M|\leq 2|P\cap M|+6$ or $P\cong 3\Delta_2$ from Scott's inequality. So we get by Pick's theorem
\begin{align*}
c_1^2=&2\mathrm{Vol}_2(F(P))=  4|F(P)\cap M|-2|\partial F(P) \cap M|-4\geq \frac{8}{3}|F(P)\cap M|-8\\=&\frac{2}{7}(c_2-48).
\end{align*}
\end{proof}

\begin{prop}
$vol(P)\leq l-\frac{5}{2}$ and $c_1^2\leq \frac{1}{2}(c_2-42).$
\end{prop}
\begin{proof}
We have at least $3$ boundary points, so by Pick's theorem $vol(P)\leq l-\frac{5}{2}$. Now we have
\begin{align*}
c_1^2=2\mathrm{Vol}_2(F(P))\leq 4 |F(P)\cap M|-10=\frac{c_2+c_1^2}{3}-14\leq \frac{1}{2}(c_2-42).
\end{align*}
\end{proof}




Our classification of Fine interiors \ref{classification} can be translated into the following theorem, result of which is visualized in figure \ref{Chern_num_drawing}. Notably, we observe many half-integral Fine interiors that realize Chern numbers not achievable by examining only Fine interiors that are lattice polygons. For instance, there are Chern numbers in areas not permitted by Scott's inequality. These Fine interiors arise in the algorithm from putting hats on hollow lattice polygons. Corresponding surfaces can be thought of as kind of a generalization of the surfaces with Chern numbers on the Noether line and Fine interiors, which are hollow polygons. Surfaces with Chern numbers on the Noether line are often referred to as \textit{Horikawa} surfaces, named after Eiji Horikawa, who studied them, for example, in \cite{Hor76}.

\begin{thm}
A pair $(c_1^2,c_2)$ is a pair of Chern numbers of a non-degenate toric hypersurface corresponding to a lattice polytope of width $2$ with up to $40$ interior lattice points and $2$ dimensional Fine interior if and only if it correspondents to a pair $(\chi,c_1^2)$ of the following table

\begin{center}
\begin{tabular}{c|c|c}
$\chi$ & \text{ number of Fine interiors} & $c_1^2 \text{ is element of }$ \\\hline
$2$&$12$&$[1,4]\cap \Z$\\
$3$&$17$&$[2,6]\cap \Z$\\
$4$&$48$&$[4,10]\cap\Z$\\
$5$&$86$&$[6,12]\cap \Z$\\
$6$&$177$&$[8,17]\cap \Z$\\
$7$&$279$&$[10,19]\cap \Z$\\
$8$&$504$&$[12,24]\cap \Z$\\
$9$&$768$&$[14,26]\cap \Z$\\
$10$&$1222$&$[16,31]\cap \Z$\\
$11$&$1850$&$[18,34]\cap \Z$\\
$12$&$2881$&$[20,39]\cap \Z$\\
$13$&$4160$&$[22,41]\cap \Z$\\
$14$&$6150$&$[24,47]\cap \Z$\\
$15$&$8480$&$[26,49]\cap \Z$\\
$16$&$12066$&$[28,54]\cap \Z$\\
$17$&$16746$&$[30,56]\cap \Z$\\
$18$&$23462$&$[32,62]\cap \Z$\\
$19$&$31601$&$[34,64]\cap \Z$\\
$20$&$42914$&$[36,69]\cap \Z$\\
$21$&$56675$&$[38,72]\cap \Z$\\
$22$&$75457$&$[40,77]\cap \Z$\\
$23$&$98713$&$[42,79]\cap \Z$\\
$24$&$129468$&$[44,85]\cap \Z$\\
$25$&$167366$&$[46,87]\cap \Z$\\
$26$&$216764$&$[48,92]\cap \Z$\\
$27$&$276569$&$[50,95]\cap \Z$\\
$28$&$352907$&$[52,100]\cap \Z$\\
$29$&$446184$&$[54,103]\cap \Z$\\
$30$&$564041$&$[56,108]\cap \Z$\\
$31$&$706531$&$[58,110]\cap \Z$\\
$32$&$884749$&$[60,116]\cap \Z$\\
$33$&$1097809$&$[62,118]\cap \Z$\\
$34$&$1362551$&$[64,124]\cap \Z$\\
$35$&$1681298$&$[66,127]\cap \Z$\\
$36$&$2071958$&$[68,131]\cap \Z$\\
$37$&$2535238$&$[70,134]\cap \Z$\\
$38$&$3099580$&$[72,139]\cap \Z$\\
$39$&$3768629$&$[74,142]\cap \Z$\\
$40$&$4578248$&$[76,147]\cap \Z$\\
\end{tabular}
\end{center}

\end{thm}

\begin{landscape}

\begin{figure}
\begin{center}
	\begin{tikzpicture}[x=0.046cm,y=0.092cm]
		
		\draw[color=black] (140,140) node { $c_1^2\leq 3c_2$ (Bogomolov-Miyaoka-Yau inequality) };
		\draw[color=black] (295,35) node { $5c_1^2\geq c_2-36$ (Noether inequality) };
		
		\draw[color=black] (415,127) node { $7c_1^2\geq 2c_2-96$};
		\draw[color=black] (415,122) node { (Scott inequality) };
		
		\draw[color=black] (200,100) node { $2c_1^2\leq c_2-42$ };		
		\draw[color=black] (150,95) node { (if the Fine interior is a lattice polygon) };
		
		\draw[color=black] (425,5) node { $c_2$ };
		\draw[color=black] (10,150) node { $c_1^2$ };


			\draw [line width=1pt,color=black] (-5,0)-- (460,0); 
			\draw [line width=1pt,color=black] (0,-3)-- (0,158); 

			\draw [line width=1pt,color=black] (54,6)-- (358,158); 
			\draw [line width=1pt,color=black] (104,16)-- (454,116); 
			\draw [line width=1pt,color=black] (54,6)-- (104,16); 
			
			
			\draw [line width=1pt,color=black] (0,0)-- (51,153);	
			\draw [line width=1pt,color=black] (36,0)-- (456,84);	


			\foreach \n in {1,...,4}{\draw [fill=black] (36-\n,\n) circle (0.5pt);}
			\foreach \n in {2,...,6}{\draw [fill=black] (48-\n,\n) circle (0.5pt);}
			\foreach \n in {4,...,10}{\draw [fill=black] (60-\n,\n) circle (0.5pt);}
			\foreach \n in {6,...,12}{\draw [fill=black] (72-\n,\n) circle (0.5pt);}
			\foreach \n in {8,...,17}{\draw [fill=black] (84-\n,\n) circle (0.5pt);}
			\foreach \n in {10,...,19}{\draw [fill=black] (96-\n,\n) circle (0.5pt);}
			\foreach \n in {12,...,24}{\draw [fill=black] (108-\n,\n) circle (0.5pt);}
			\foreach \n in {14,...,26}{\draw [fill=black] (120-\n,\n) circle (0.5pt);}
			\foreach \n in {16,...,31}{\draw [fill=black] (132-\n,\n) circle (0.5pt);}
			\foreach \n in {18,...,34}{\draw [fill=black] (144-\n,\n) circle (0.5pt);}
			\foreach \n in {20,...,39}{\draw [fill=black] (156-\n,\n) circle (0.5pt);}
			\foreach \n in {22,...,41}{\draw [fill=black] (168-\n,\n) circle (0.5pt);}
			\foreach \n in {24,...,47}{\draw [fill=black] (180-\n,\n) circle (0.5pt);}
			\foreach \n in {26,...,49}{\draw [fill=black] (192-\n,\n) circle (0.5pt);}
			\foreach \n in {28,...,54}{\draw [fill=black] (204-\n,\n) circle (0.5pt);}
			\foreach \n in {30,...,56}{\draw [fill=black] (216-\n,\n) circle (0.5pt);}
			\foreach \n in {32,...,62}{\draw [fill=black] (228-\n,\n) circle (0.5pt);}
			\foreach \n in {34,...,64}{\draw [fill=black] (240-\n,\n) circle (0.5pt);}
			\foreach \n in {36,...,69}{\draw [fill=black] (252-\n,\n) circle (0.5pt);}
			\foreach \n in {38,...,72}{\draw [fill=black] (264-\n,\n) circle (0.5pt);}
			\foreach \n in {40,...,77}{\draw [fill=black] (276-\n,\n) circle (0.5pt);}
			\foreach \n in {42,...,79}{\draw [fill=black] (288-\n,\n) circle (0.5pt);}
			\foreach \n in {44,...,85}{\draw [fill=black] (300-\n,\n) circle (0.5pt);}
			\foreach \n in {46,...,87}{\draw [fill=black] (312-\n,\n) circle (0.5pt);}
			\foreach \n in {48,...,92}{\draw [fill=black] (324-\n,\n) circle (0.5pt);}
			\foreach \n in {50,...,95}{\draw [fill=black] (336-\n,\n) circle (0.5pt);}
			\foreach \n in {52,...,100}{\draw [fill=black] (348-\n,\n) circle (0.5pt);}
			\foreach \n in {54,...,103}{\draw [fill=black] (360-\n,\n) circle (0.5pt);}
			\foreach \n in {56,...,108}{\draw [fill=black] (372-\n,\n) circle (0.5pt);}
			\foreach \n in {58,...,110}{\draw [fill=black] (384-\n,\n) circle (0.5pt);}
			\foreach \n in {60,...,116}{\draw [fill=black] (396-\n,\n) circle (0.5pt);}
			\foreach \n in {62,...,118}{\draw [fill=black] (408-\n,\n) circle (0.5pt);}
			\foreach \n in {64,...,124}{\draw [fill=black] (420-\n,\n) circle (0.5pt);}
			\foreach \n in {66,...,127}{\draw [fill=black] (432-\n,\n) circle (0.5pt);}
			\foreach \n in {68,...,131}{\draw [fill=black] (444-\n,\n) circle (0.5pt);}
			\foreach \n in {70,...,134}{\draw [fill=black] (456-\n,\n) circle (0.5pt);}
			\foreach \n in {72,...,139}{\draw [fill=black] (468-\n,\n) circle (0.5pt);}
			\foreach \n in {74,...,142}{\draw [fill=black] (480-\n,\n) circle (0.5pt);}
			\foreach \n in {76,...,147}{\draw [fill=black] (492-\n,\n) circle (0.5pt);}
			

	\end{tikzpicture}
	\caption{The Chern numbers of the classified surfaces of general type}\label{Chern_num_drawing}
\end{center}

\end{figure}
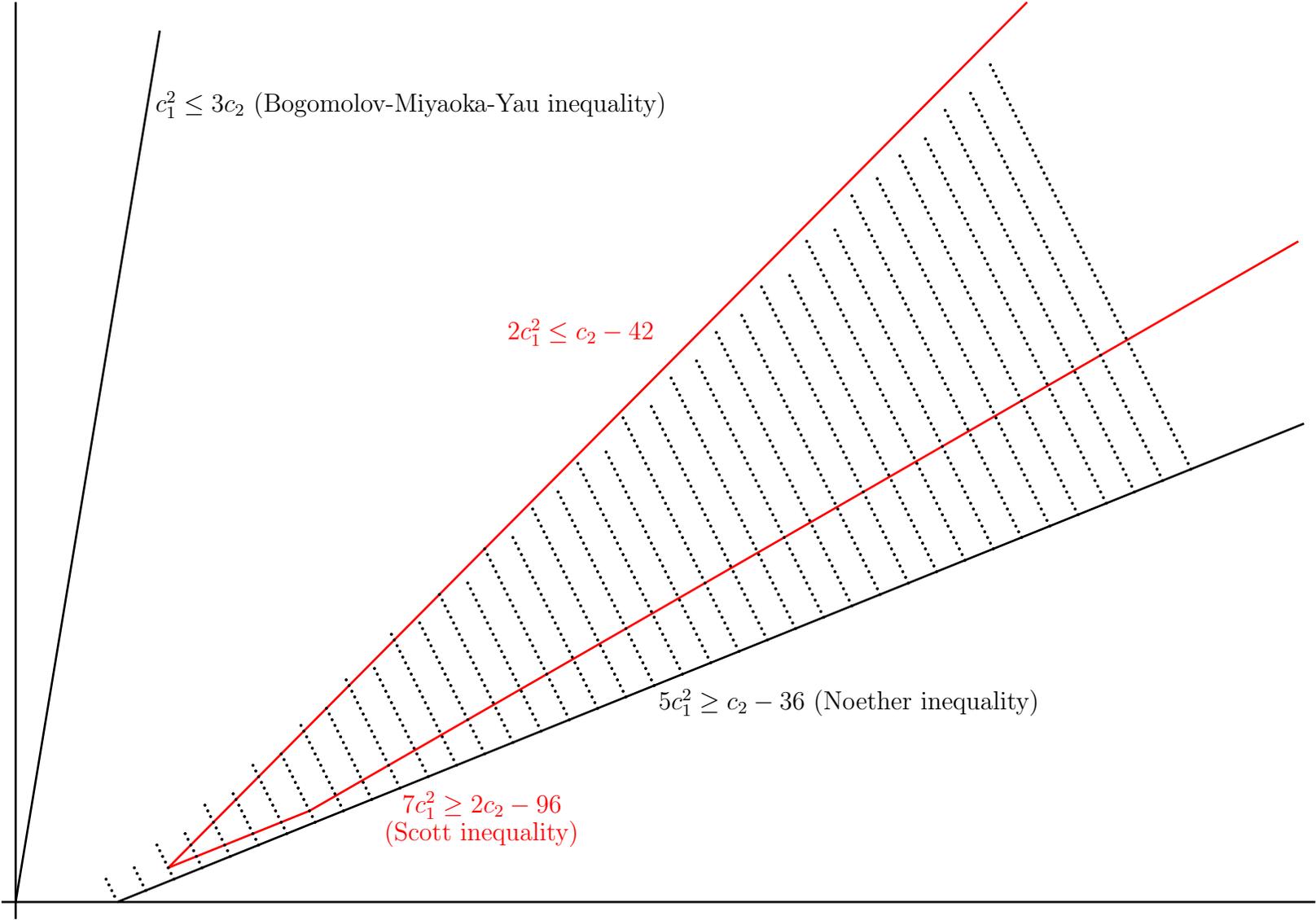

\end{landscape}

\textbf{Acknowledgements.} I would like to thank Victor Batyrev for introducing me to the Fine interior. I would also like to express my gratitude to Alexander M. Kasprzyk for his interest in my work and our discussions on Fine interiors, and Daniel Hättig for inspiring me to study half-integral polygons.

\end{document}